\documentclass[fleqn,12pt]{amsart} 
\usepackage{amssymb}
\usepackage{enumerate}  
\usepackage{graphicx}  
\usepackage{epstopdf}
\usepackage{comment}
\usepackage{array}

 \makeatletter
    
    \@addtoreset{equation}{section}
  \makeatother

\theoremstyle{plain} 
\newtheorem{theorem}{\indent\sc Theorem}[section] 

\newtheorem{corollary}[theorem]{\indent\sc Corollary}
\newtheorem{proposition}[theorem]{\indent\sc Proposition}

\theoremstyle{definition} 
\newtheorem{definition}[theorem]{\indent\sc Definition}
\newtheorem{remark}[theorem]{\indent\sc Remark}
\newtheorem{example}[theorem]{\indent\sc Example}

\DeclareMathOperator{\im}{im}
\DeclareMathOperator{\kernel}{ker}
\DeclareMathOperator{\sgn}{sgn}

\begin{document}

\title[An estimate of the first non-zero eigenvalue of the Laplacian]{An estimate of the first non-zero eigenvalue of the Laplacian by the Ricci curvature on edges of graphs} 

\author[T. Yamada]{Taiki Yamada} 


\footnote{ 
2010 \textit{Mathematics Subject Classification}.
Primary 05C12; Secondary 52C99, 35J05.
}
\keywords{ 
Graph theory, Coarse Ricci curvature, Simplicial complex, Graph Laplacian
}
\address[T. Yamada]{ 
Mathematical Institute in Tohoku University \endgraf
Sendai 980-8578 \endgraf
Japan
}
\email{mathyamada@dc.tohoku.ac.jp}


\maketitle

\begin{abstract}
We define the distance between edges of graphs and study the coarse Ricci curvature on edges. We consider the Laplacian on edges based on the Jost-Horak's definition of the Laplacian on simplicial complexes. As one of our main results, we obtain an estimate of the first non-zero eigenvalue of the Laplacian by the Ricci curvature for a regular graph. 
\end{abstract}

\section{Introduction}
 The Ricci curvature is one of the most important concepts in Riemannian geometry. There are some definitions of the generalized Ricci curvature and Ollivier's coarse Ricci curvature is one of them (see \cite{Ol1}, \cite{Ol2}). It is formulated by the Wasserstein distance on a metric space $(X, d)$ with a random walk $m=\left\{m_{x} \right\}_{x \in X}$, where $m_{x}$ is a probability measure on $X$. The coarse Ricci curvature is defined as, for two distinct points $x, y \in X$, 
	\begin{eqnarray*}
	\kappa(x, y) := 1 - \cfrac{W(m_{x}, m_{y})}{d(x, y)},
	\end{eqnarray*}
  where $W$($m_{x}, m_{y}$) is the $1$-Wasserstein distance between $m_{x}$ and $m_{y}$. This definition was applied to graphs around  2010 and many researchers are focused on this. In 2010, Lin, Lu and Yau \cite{Yau1} defined the Ricci curvature on vertices of graphs by using the coarse Ricci curvature of the lazy random walk and they studied the Ricci curvature of product space of graphs and random graphs. In 2012, Jost and Liu \cite{Jo2} studied the relation between the Ricci curvature and the  local clustering efficient. Recently, the Ricci curvature on graphs was applied to directed graphs \cite{Yamada} and internet topology \cite{Ni} and so on. In this paper, we study the Ricci curvature on edges of graphs.\\
　The study of the graph Laplacian has a long and prolific history. In 1847, Kirchhoff \cite{Kir} was first defined the graph Laplacian on real valued functions. In the early 1970s, Fiedler \cite{Fie} studied the relation between the first non-zero eigenvalue and the connectivity of a graph. After he was focused on the graph Laplacian, there has been many papers about the graph Laplacian and its spectrum. In 1935, Bottema \cite{Bo} introduced the normalized the graph Laplacian and studied the relation between the normalized graph Laplacian and the transition probability operator on graphs. Moreover, in 2009, Ollivier \cite{Ol1} obtained an estimate of the first non-zero eigenvalue of the normalized graph Laplacian by a lower bound of the coarse Ricci curvature.\\
　The graph Laplacian was generalized to simplicial complexes by Eckmann \cite{Eck} as the discrete version of the Hodge theory. This graph Laplacian is called the {\em higher order combinatorial Laplacian}. In 2013, Jost and Horak \cite{Ho} developed a general framework for Laplace operators defined by the combinatorial structure of a simplicial complex. This includes all graph Laplacians defined before, and they showed in \cite{Ho} that the spectrum of the graph Laplacian on vertices coincides with that of the graph Laplacian on edges. Considering the Ricci curvature on edges, we obtain the following theorem.
	\begin{theorem}
	\label{main0}
	Let $G$ be a $d$-regular graph and $\lambda_{1}$ the first non-zero eigenvalue of the Laplacian on edges. Suppose that $\kappa(e, e') \geq \kappa$ for any $e, e' \in E$ and for a real number $\kappa > 0$. Then we have
		\begin{eqnarray}
		\label{main0-0}
		\lambda_{1} \geq \kappa + \cfrac{2}{d} - 1,
		\end{eqnarray}
	where $\kappa(e, e')$ is the Ricci curvature on $e$ and $e'$ and the Laplacian here is defined by Jost-Horak $($see Definition \ref{laplacian}$).$
	\end{theorem}
　The estimate of Theorem \ref{main0} is optimal. In fact, any star graph satisfies equality of \eqref{main0-0} (see Remark \ref{main0-1}).\\
　This paper is organized as follows. In the 2nd section, we define the Ricci curvature on edges of a graph and we prove some properties of the Ricci curvature. In the 3rd section, for a $1$-dimensional abstract simplicial complex, we prove some properties of the Laplacian on edges defined in \cite{Ho}. Moreover, we prove Theorem \ref{main0}. In the 4th section, we treat some examples of graphs, and calculate the Ricci curvature and the eigenvalue of the Laplacian. In the 5th section, considering weighted graphs, and we generalize the results proved in the previous sections. 
\section*{acknowledgment} 
The author thank Professor J\"{u}rgen Jost for useful suggestion. When I went to the Max Planck Institute, he gave me some helpful and accurate advices to this topic.

\section{Properties of Ricci curvature on edges of graphs}
 Before we define the Ricci curvature on edges, we present some preliminaries.\\
 　In this paper, $G=(V, E)$ is an undirected connected simple finite graph, where $V$ is the set of the vertices and $E$ the set of edges. That is,
	\begin{enumerate}
	\item for any two vertices, there exists a path connecting them,
	\item there exists no loop and no multiple edge,
	\item the number of vertices and edges is finite.
	\end{enumerate}
  For $x, y \in V$, we write $(x, y)$ as an edge connecting $x$ and $y$.  We denote the set of vertices of $G$ by $V(G)$ and the set of edges by $E(G)$. 
	\begin{definition} 
	\label{path}
		\begin{enumerate}
		\item Two distinct edges $e$ and $e'$ are {\em connected} if they have a common vertex.
		\item A {\em path} connecting two edges $e$ and $e'$ is a sequence of edges $\left\{e_{i} \right\}_{i=0}^{n}$ such that $e_{i}$ and $e_{i+1}$ are connected for $0 \leq i \leq n-1$ and $e_{0}=e$, $e_{n} = e'$. We call $n$ the {\em length} of the path. 
		\end{enumerate}
	\end{definition}
  The {\em distance} $d(e, e')$ between two edges $e, e' \in E$ is defined to be the length of a shortest path connecting $e$ and $e'$. \\
　 For any $x \in V$, the {\em neighborhood} of $x$ is defined as $\Gamma(x) := \left\{y \in V \mid (x, y) \in E \right\}$ and the {\em degree} of $x$, denoted by $d^{V}_{x}$, is the number of edges connecting $v$, i.e., $d^{V}_{x} = |\Gamma(x)|$. Similarly, we define the degree of an edge as following.
	\begin{definition}
	For any $e \in E$, the {\em neighborhood} of $e$ is defined as
		\begin{eqnarray*}
		\Gamma(e) := \left\{\bar{e} \in V \mid e\ \mathrm{and}\ \bar{e}\ \mathrm{are\ connected\ and}\ e \neq \bar{e}. \right\}.
		\end{eqnarray*}
	The {\em degree} of $e$, denoted by $d_{e}$, is the number of edges connecting $e$, i.e., $d_{e} = |\Gamma(e)|$. Moreover, $G$ is a $d$-\emph{regular graph} if every edge has the same degree $d$.
	\end{definition}
 With these definitions, we define the Ricci curvature on edges.
	\begin{definition}
	For any edge $e \in E$, we define a probability measure $m_{e}$ on $E$ by
		\begin{eqnarray}
		\label{measure}
		m_{e}(\bar{e}) = 
			\begin{cases}
			\cfrac{1}{d_{e}}, &\mathrm{if}\ \bar{e} \in \Gamma(e), \\
			0, & \mathrm{\mathrm{otherwise}}.
			\end{cases}
		\end{eqnarray}
	This defines a random walk $m=\left\{m_{x} \right\}_{x \in X}$.
	\end{definition} 

	\begin{definition}
	For two probability measures $\mu$ and $\nu$ on $E$, the 1-Wasserstein distance between $\mu$ and $\nu$ is written as
		\begin{eqnarray*}
		W(\mu, \nu) = \inf_{A} \sum_{\bar{e}, \bar{e'} \in E}A(\bar{e}, \bar{e'})d(\bar{e}, \bar{e'}),
		\end{eqnarray*}
	where $A : E \times E \to [0, 1]$ runs over all maps satisfying 
		\begin{eqnarray*}
			\begin{cases}
			\sum_{\bar{e'} \in E}A(\bar{e}, \bar{e'}) = \mu(\bar{e}),\\
			\sum_{\bar{e} \in E}A(\bar{e}, \bar{e'}) = \nu(\bar{e'}).
			\end{cases}
		\end{eqnarray*}
	Such a map $A$ is called a {\em coupling} between $\mu$ and $\nu$. 
	\end{definition}
	\begin{remark}
	There exists a coupling $A$ that attains the Wasserstein distance (see \cite{Le}, \cite{Vi1} and \cite{Vi2}).
	\end{remark}
   One of the most important properties of the Wasserstein distance is the Kantorovich-Rubinstein duality as stated as follows.
	\begin{proposition}[Kantorovich, Rubinstein]
	\label{kantoro}
	For two probability measures $\mu$ and $\nu$ on $E$, the 1-Wasserstein distance between $\mu$ and $\nu$ is written as
		\begin{eqnarray*}
		W(\mu, \nu) = \sup_{f : 1-Lip} \sum_{\bar{e} \in E} f(\bar{e})(m_{e}(\bar{e}) - m_{e'}(\bar{e})),
		\end{eqnarray*}
	where the supremum is taken over all functions on $G$ that satisfy $|f(e) - f(e')| \leq d(e, e')$ for any $e, e' \in E$.
	\end{proposition}
	\begin{definition}
	For any two distinct edges $e, e' \in E$, the {\em Ricci curvature} of $e$ and $e'$ is defined as
		\begin{eqnarray*}
		\kappa(e, e') = 1 - \cfrac{W(m_{e}, m_{e'})}{d(e, e')}.
		\end{eqnarray*}
	\end{definition}
  If we apply some properties of the Ricci curvature on vertices proved in \cite{Jo0} and \cite{Jo2} to the Ricci curvature on edges, then we obtain the following results.
	\begin{proposition}
	\label{pair}
	If $\kappa(e, e') \geq \kappa_{0}$ for any edges $e$ and $e'$ with $d(e, e') = 1$, then $\kappa(e, e') \geq \kappa_{0}$ for any pair of edges $(e, e')$.
	\end{proposition}
	\begin{theorem}
	\label{Riccilower}
  	For any edges $e$ and $e'$ with $d(e, e') = 1$, we have
    		\begin{eqnarray*}
     		\kappa(e, e') &\geq& - 2 \left(1 - \cfrac{1}{d_{e}} - \cfrac{1}{d_{e'}} \right)_{+}\\
		&=& 
			\begin{cases}
			- 2  + \cfrac{2}{d_{e}} + \cfrac{2}{d_{e'}} ,& \mathrm{if}\ d_{e}>1\ \mathrm{and}\ d_{e'}>1\\
			0,& otherwise.
			\end{cases}
		\end{eqnarray*} 
	\end{theorem}
	\begin{theorem}
	\label{Ricciupper}
  	For any edges $e \in E$ and $e' \in \Gamma(e)$, we have
  		\begin{eqnarray*}
   		\kappa(e, e') \leq \cfrac{|\Gamma(e) \cup \Gamma(e')|}{\max \left\{d_{e}, d_{e'}\right\}}.
  		\end{eqnarray*}
	\end{theorem}

\section{Properties of the Laplacian on edges in the case $\dim G=1$}
 First we review the definition of the cohomology. Let $K$ be an abstract simplicial complex and $S_{i}(K)$ the set of all $i$-faces of $K$. We say that a face $F$ is {\em oriented} if we choose an ordering on its vertices and write $[F]$. \\
　The $i$-th chain group $C_{i}(K, \mathbb{R})$ of $K$ with coefficients in $\mathbb{R}$ is a vector space over the real field $\mathbb{R}$ with basis $B_{i}(K, \mathbb{R})=\left\{ [F] \mid F \in S_{i}(K) \right\}$, and the $i$-th cochain group $C^{i}(K, \mathbb{R})$ is defined as the dual of the $i$-th chain group.
	\begin{definition}
	For the cochain groups, the {\em simplicial coboundary maps} $\delta_{i} : C^{i}(K, \mathbb{R}) \to C^{i+1}(K, \mathbb{R})$ are defined by
	\begin{eqnarray*}
(\delta_{i} f)([v_{0}, \cdots, v_{i + 1}]) = \sum_{j = 0}^{i + 1} (- 1)^{j} f([v_{0}, \cdots, \hat{v_{j}}, \cdots, v_{i + 1}])
	\end{eqnarray*}
	for $f \in C^{i}(K, \mathbb{R})$, where $\hat{v_{j}}$ means that the vertex $v_{j}$ has been omitted.
	\end{definition} 
Note that the one-dimensional vector space $C^{-1}(K, \mathbb{R})$ is generated by the identity function on the empty simplex. The $\delta_{i}$ are the connecting maps in the {\em augmented cochain complex} of $K$ with coefficients in $\mathbb{R}$, i.e, the sequence of the vector spaces with the linear transformations
	\begin{eqnarray*}
 	\cdots \xleftarrow{\delta_{i+1}} C^{i+1}(K, \mathbb{R}) \xleftarrow{\delta_{i}} C^{i}(K, \mathbb{R}) \xleftarrow{\delta_{i-1}} \cdots \xleftarrow{\delta_{1}} C^{1}(K, \mathbb{R}) \xleftarrow{\delta_{0}} C^{0}(K, \mathbb{R}) \xleftarrow{\delta_{-1}} C^{-1}(K, \mathbb{R}) \leftarrow 0
	\end{eqnarray*} 
It is easy to show that $\delta_{i} \delta_{i-1} = 0$, thus the image of $\delta_{i-1}$ is contained in the kernel of $\delta_{i}$. We define the reduced cohomology group by
	\begin{eqnarray*}
	\tilde{H}^{i}(K, \mathbb{R}) := \ker \delta_{i} / \im \delta_{i-1}.
	\end{eqnarray*}
In order to define the Laplacian, we define the boundary of the oriented face and the inner product.
	\begin{definition}
	Let $F'=\left\{v_{0}, \cdots, v_{i+1}\right\}$ be an $(i+1)$-face of $K$ and $F = \left\{v_{0}, \cdots, \hat{v_{j}}, \cdots, v_{i+1}\right\}$ an $i$-face of $F'$. The {\em boundary of the oriented face} $[F']$ is 
		\begin{eqnarray*}
 		\partial [F'] = \sum_{j}(-1)^{j}[v_{0}, \cdots, \hat{v_{j}}, \cdots, v_{i+1}],
		\end{eqnarray*}
	and the sign of $[F]$ in the boundary of $[F']$ is denoted by sgn($[F], \partial [F']$) and is equal to $(-1)^{j}$.
	\end{definition}
	\begin{definition}
	The {\em inner product} on the space $C^{i}(K, \mathbb{R})$ is
		\begin{eqnarray}
		\label{inner}
		(f, g)_{C^{i}} = \sum_{F \in S_{i}(K)} w(F) f([F]) g([F])
		\end{eqnarray}
	for $f$, $g \in C^{i}(K, \mathbb{R})$, where $w: \bigcup_{i = 0} S_{i}(K) \to \mathbb{R}^{+}$ is a function with $w(\emptyset) = 0$. We call $w$ the weight function on $K$.
	\end{definition}
For the innner product on $C^{i}(K, \mathbb{R})$, the adjoint $\delta^{*}_{i} : C^{i + 1}(K, \mathbb{R}) \to C^{i}(K, \mathbb{R})$ of the coboundary operator $\delta_{i}$ is defined by
	\begin{eqnarray*}
	(\delta_{i} f_{1}, f_{2})_{C^{i+1}} = ( f_{1},\delta^{*}_{i} f_{2})_{C^{i}},
	\end{eqnarray*}
for every $f_{1} \in C^{i}(K, \mathbb{R})$ and $f_{2} \in C^{i + 1}(K, \mathbb{R})$. The adjoint operator $\delta_{i}^{*}$ is expressed as
	\begin{eqnarray*}
	(\delta_{i}^{*}f)(F) = \cfrac{1}{w(F)}\sum_{F' \in S_{i+1}(K): F \in \partial F'} \sgn ([F], \partial 
[F'])f([F'])
	\end{eqnarray*}
	for $f \in C^{i+1}(K, \mathbb{R})$.
	\begin{definition}[Horak, Jost, \cite{Ho}]
	\label{laplacian}
	We define the Laplace operators on $C^{i}(K, \mathbb{R})$ as follows.
		\begin{enumerate}
		\item The $i$-{\em dimensional combinatorial up Laplace operator} or {\em simply $i$-up Laplace operator} is defined by
 			\begin{eqnarray*}
  			\mathcal{L}^{up}_{i}(K) := \delta^{*}_{i} \delta_{i}.
 			\end{eqnarray*}
		\item The {\em $i$-dimensional combinatorial down Laplace operator} or {\em simply $i$-down Laplace operator} is defined by
 			\begin{eqnarray*}
  			\mathcal{L}^{down}_{i}(K) := \delta_{i-1} \delta^{*}_{i-1}.
 			\end{eqnarray*}
		\item The {\em $i$-dimensional combinatorial Laplace operator} or {simply $i$-Laplace operator} is defined by
 			\begin{eqnarray*}
  			\mathcal{L}_{i}(K) := \mathcal{L}^{up}_{i}(K) + \mathcal{L}^{down}_{i}(K) =  \delta^{*}_{i} \delta_{i} + \delta_{i-1} \delta^{*}_{i-1}.
 			\end{eqnarray*}
		\end{enumerate}
	\end{definition}
	\begin{remark}
	\label{eigenvalue}
	Since their operators are all self-adjoint and non-negative, the eigenvalues are real and non-negative. In addition, since $\delta_{i} \delta_{i-1} = 0$ and $\delta^{*}_{i-1} \delta^{*}_{i} = 0$, we have
		\begin{eqnarray*}
		\im \mathcal{L}^{down}_{i} \subset \kernel \mathcal{L}^{up}_{i},\\
		\im \mathcal{L}^{up}_{i} \subset \kernel \mathcal{L}^{down}_{i}.
		\end{eqnarray*}
	Thus, $\lambda$ is a non-zero eigenvalue of $\mathcal{L}_{i}$ if and only if it is a non-zero eigenvalue of $\mathcal{L}^{up}_{i}$ or $\mathcal{L}^{down}_{i}$. 
	\end{remark}
We represent the Laplace operators by the matrix form. Let $D_{i}$ be the matrix corresponding to the operator $\delta_{i}$, $D_{i}^{T}$ its transpose and $W_{i}$ the diagonal matrix representing their scalar product on $C^{i}$, then the operators $\mathcal{L}^{up}_{i}(K)$ and $\mathcal{L}^{down}_{i}(K)$ are expressed as 
	\begin{eqnarray*}
	\mathcal{L}^{up}_{i}(K) = W_{i}^{-1}D_{i}^{T}W_{i+1}D_{i}
	\end{eqnarray*}
and
	\begin{eqnarray*}
	\mathcal{L}^{down}_{i}(K) = D_{i-1}W_{i-1}^{-1}D_{i-1}^{T}W_{i}.
	\end{eqnarray*}
　Since a graph is a $1$-dimensional abstract simplicial complex and we have $W_{-1} = w(\emptyset) = 0$, we obtain $\mathcal{L}^{down}_{0}(G) = 0$ and $\mathcal{L}^{up}_{1}(G) = 0$. As the result, it is sufficient to consider the following Laplace operator:
	\begin{eqnarray*}
	\mathcal{L}_{0}(G) = \mathcal{L}^{up}_{0}(G) = W_{0}^{-1}D_{0}^{T}W_{1}D_{0}
	\end{eqnarray*}
and
	\begin{eqnarray*}
	\mathcal{L}_{1}(G) = \mathcal{L}^{down}_{1}(G) = D_{0}W_{0}^{-1}D_{0}^{T}W_{1}.
	\end{eqnarray*}
By the matrix form, $\lambda$ is a non-zero eigenvalue of $\mathcal{L}_{0}(G)$ if and only if it is a non-zero eigenvalue of $\mathcal{L}_{1}(G)$.
	\begin{remark}
If $W_{0}$ and $W_{1}$ are the identity matrix, then the $0$-Laplace operator $\mathcal{L}_{0}(G)$ corresponds to the graph Laplacian, so the operator does not depend on the orientation of faces. On the other hand, the $1$-Laplace operator $\mathcal{L}_{1}(G)$ depends on the orientation of faces. However, the $1$-Laplacian has the same spectrum as the $0$-Laplacian, except for the multiplicity of the eigenvalue $0$, so the spectrum of $\mathcal{L}_{1}(G)$ does not depend on the orientation of faces.
	\end{remark}
If $W_{1}$ is the identity matrix and if $w(v_{i})$ is the degree of $v_{i}$, then the obtained operator is called the {\em normalized combinatorial Laplace operator} and is denoted by $\Delta_{0}$.
On the other hand, we want to obtain an estimate of the first non-zero eigenvalue of the $1$-Laplacian by the Ollivier's Ricci curvature. In order to do so, we assume that $W_{0}$ is an identity matrix and $w(e_{i}) = 1/d_{e_{i}}$, and the obtained operator is denoted by $\mathcal{L'}_{1}$.\\
Before we obtain an estimate of the eigenvalues of the Laplacian, we decompose the Laplacian into the diagonal component and the other part
	\begin{eqnarray*}
	\mathcal{L'}_{1} = \mathrm{Diag}(\mathcal{L'}_{1}) + L^{down},
	\end{eqnarray*}
	where 
	\begin{eqnarray*}
	(L^{down}f)(e) = \sum_{e' \in E(G): v= e \cap e'} \mathrm{sgn}([v], \partial [e]) \mathrm{sgn}([v], \partial [e'])m_{e}(e') f(e') \cfrac{d_{e}}{d_{e'}}.
	\end{eqnarray*}
　Ollivier proved an estimate of the first non-zero eigenvalue of the normalized graph Laplace operator $\Delta_{0}$ by the Ollivier's Ricci curvature as follows.
	\begin{theorem}[Ollivier, \cite{Ol1}]
	On $(V(G), d, m)$, if $\kappa(x, y) \geq \kappa$ for any $(x, y) \in E$ and for a positive real number $\kappa$, then the first non-zero eigenvalue of the normalized graph Laplace operator $\Delta_{0}$ satisfies
		\begin{eqnarray*}
		\lambda_{1} \geq \kappa,
		\end{eqnarray*}
	\end{theorem}
However, this theorem does not hold for $\mathcal{L'}_{1}$ in general. A star graph is such a counter example (see Example \ref{star}).
\begin{proof}[Proof of Theorem \ref{main0}]
Let $f$ be an eigenfunction with respect to $\lambda_{1}$. Fix any two edges $e=(x, y), e'=(y, z) \in E(G)$ with $d(e, e')=1$. \\
　Case 1. If $e$ and $e'$ are not contained in any triangle, then we orient edges $\bar{e} \in \Gamma(e) \cup \Gamma(e')$ as follows (see Figure 1).
	\begin{enumerate}
		\item If $\bar{e} = e$, or if $\bar{e}$ and $e$ are connected by $x$, then we orient $\bar{e}$ with $\sgn ([x], \partial[\bar{e}]) = -1$.
		\item If $\bar{e}$ and $e$ are connected by $y$, or if $\bar{e}$ and $e'$ are connected by $y$, then we orient $\bar{e}$ with $\sgn ([y], \partial[\bar{e}]) = 1$.
		\item If $\bar{e} = e'$, or if $\bar{e}$ and $e$ are connected by $z$, then we orient $\bar{e}$ with $\sgn ([z], \partial[\bar{e}]) = -1$.
	\end{enumerate} 
	\begin{figure}[h]
     	\label{Ori1} 
     		\begin{center} 
     		\includegraphics[scale=0.6]{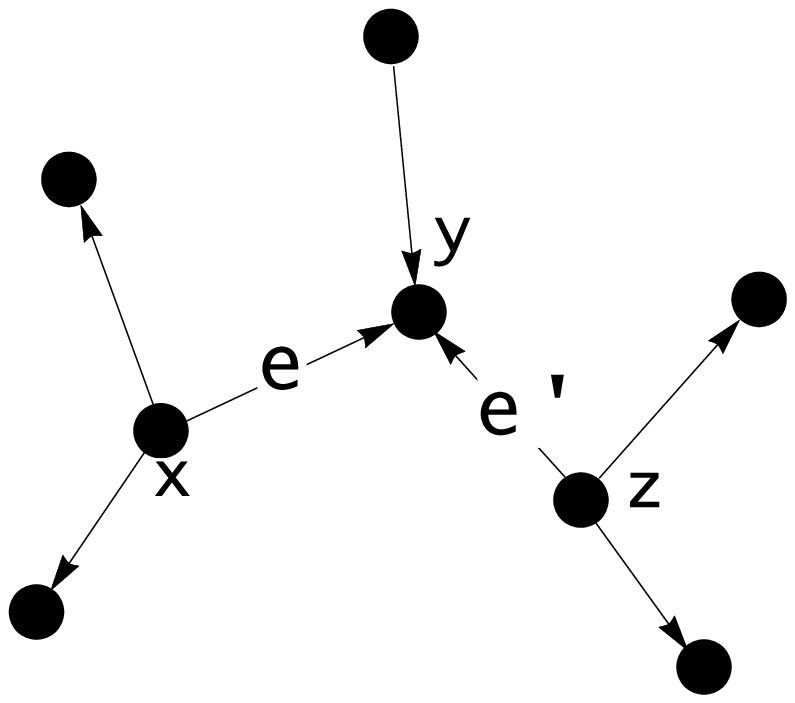}
     		\caption{Orientation of $\Gamma(e) \cup \Gamma(e')$}
     		\end{center}
    	\end{figure}
Then we obtain
	\begin{eqnarray}
	\label{case1-equality1}
		\begin{cases}
		(L^{down}f)(e) = \sum_{\bar{e} \in \Gamma(e)}m_{e}(\bar{e})f(\bar{e}) =   \lambda_{1} f(e) - \cfrac{2}{d} f(e),\\ 
		(L^{down}f)(e') = \sum_{\bar{e'} \in \Gamma(e')}m_{e'}(\bar{e'})f(\bar{e'}) =   \lambda_{1} f(e') - \cfrac{2}{d} f(e'). 
		\end{cases}
	\end{eqnarray}
If $f$ is a constant function, then $\lambda_{1}=1+2/d$. In this case, there is nothing to prove. In fact, by the definition of the Ricci curvature, we see $\kappa \leq 1$. So, we have
	\begin{eqnarray*}
	\kappa + \cfrac{2}{d} - 1 \leq \cfrac{2}{d} < \lambda_{1}. 
	\end{eqnarray*}
Thus we assume that $f$ is not a constant function and by scaling $f$ if necessarily, 
	\begin{eqnarray}
	\label{lip}
	\sup_{e, e' \in E : d(e, e')=1}|f(e) - f(e')| = 1.
	\end{eqnarray}
By Proposition \ref{kantoro} and (\ref{case1-equality1}), we have
	\begin{eqnarray*}
	d(e, e') ( 1 - \kappa ) &\geq& W(m_{e} , m_{e'}) \\ 
	&\geq& \sum_{\bar{e} \in \Gamma(e) \cup \Gamma(e')} f(\bar{e})(m_{e}(\bar{e}) - m_{e'}(\bar{e})) \\
	&=&  \sum_{\bar{e} \in \Gamma(e)}m_{e}(\bar{e})f(\bar{e}) - \sum_{\bar{e'} \in \Gamma(e')}m_{e'}(\bar{e'})f(\bar{e'}) \\
	&\geq&  \lambda_{1} f(e) - \cfrac{2}{d} f(e) - \lambda_{1} f(e') + \cfrac{2}{d} f(e') \\
	&=& - \left(\cfrac{2}{d} - \lambda_{1} \right) (f(e) -f(e')).
	\end{eqnarray*} 
By the symmetry of $e$ and $e'$, we obtain
	\begin{eqnarray*}
	|f(e) -f(e')| \leq \cfrac{1 - \kappa}{|2/d - \lambda_{1}|}.
	\end{eqnarray*}
Then, the above inequality holds for any edge, which together with (\ref{lip}), implies
	\begin{eqnarray*}
	\lambda_{1} &\geq& \kappa + \cfrac{2}{d} - 1.
	\end{eqnarray*}

　Case 2. If $e$ and $e'$ are contained in a triangle, then we put $e_{0} :=(x, z)$. We orient the edge $e_{0}$ with $\sgn([x], \partial[e_{0}])=1$, and orient other edges $\bar{e} \in (\Gamma(e) \cup \Gamma(e')) \setminus \left\{e_{0} \right\}$ in the same way as Case 1 (see Figure 2). 
 	\begin{figure}[h]
     	\label{Ori3} 
     		\begin{center} 
    		\includegraphics[scale=0.7]{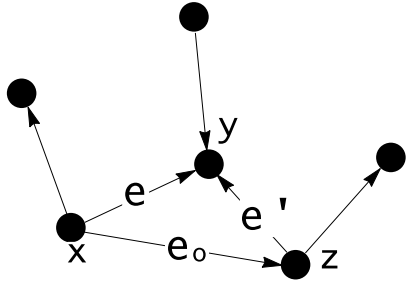}
     		\caption{Orientation of $\Gamma(e) \cup \Gamma(e')$}
     		\end{center}
    	\end{figure}
	
	Then we obtain
	\begin{eqnarray}
	\label{case2-equality1}
		\begin{cases}
		(L^{down}f)(e) = \sum_{\bar{e} \in \Gamma(e)}m_{e}(\bar{e})f(\bar{e}) =   \lambda_{1} f(e) - \cfrac{2}{d} f(e),\\ 
		(L^{down}f)(e') = \sum_{\bar{e'} \in \Gamma(e') \setminus \left\{e_{0} \right\} }m_{e'}(\bar{e'})f(\bar{e'}) - \cfrac{1}{d}f(e_{0})=   \lambda_{1} f(e') - \cfrac{2}{d} f(e'), 
		\end{cases}
	\end{eqnarray}
If $f$ is a constant function, then $\lambda_{1}=0$. This is a contradiction. Thus $f$ is not a constant function. By scaling $f$ if necessarily, we assume \eqref{lip}. We focus on $e_{0}$. By (\ref{case2-equality1}), we have
	\begin{eqnarray*}
	\lambda_{1} f(e_{0}) &=& \cfrac{2}{d}f(e_{0}) + \cfrac{1}{d}f(e) + \sum_{\bar{e} \in \Gamma(e) \setminus (\Gamma(e') \cup \left\{e' \right\})}m_{e}(\bar{e})f(\bar{e}) - \cfrac{1}{d}f(e') - \sum_{\bar{e'} \in \Gamma(e') \setminus (\Gamma(e) \cup \left\{e \right\})}m_{e'}(\bar{e'})f(\bar{e'})\\
	&=& \cfrac{2}{d}f(e_{0}) + \cfrac{2}{d}f(e) + \sum_{\bar{e} \in \Gamma(e) \setminus \Gamma(e') }m_{e}(\bar{e})f(\bar{e}) - \cfrac{2}{d}f(e') - \sum_{\bar{e'} \in \Gamma(e') \setminus \Gamma(e)}m_{e'}(\bar{e'})f(\bar{e'})\\
	&=& \cfrac{2}{d}f(e_{0}) + \lambda_{1} f(e) - \lambda_{1} f(e_{0}) - \cfrac{2}{d}f(e_{0})\\
	&=& \lambda_{1} (f(e) - f(e')).
	\end{eqnarray*}
Thus we obtain 
	\begin{eqnarray}
	\label{case2-equality2}
	f(e_{0}) = f(e) - f(e').
	\end{eqnarray}
By Proposition \ref{kantoro} and (\ref{case2-equality2}), we have 
	\begin{eqnarray*}
	d(e, e') ( 1 - \kappa ) &\geq& W(m_{e} , m_{e'}) \\ 
	&\geq& \sum_{\bar{e} \in (\Gamma(e) \cup \Gamma(e'))} f(\bar{e})(m_{e}(\bar{e}) - m_{e'}(\bar{e})) \\
	&=&  \sum_{\bar{e} \in \Gamma(e) \setminus \left\{e_{0} \right\}}m_{e}(\bar{e})f(\bar{e}) - \sum_{\bar{e'} \in \Gamma(e') \setminus \left\{e_{0} \right\}}m_{e'}(\bar{e'})f(\bar{e'}) \\
	&=&  \lambda_{1} f(e) - \cfrac{2}{d} f(e) - \cfrac{1}{d}f(e_{0}) - \lambda_{1} f(e') + \cfrac{2}{d} f(e') - \cfrac{1}{d}f(e_{0})\\
	&\geq&  \lambda_{1} f(e) - \cfrac{2}{d} f(e) - \lambda_{1} f(e') + \cfrac{2}{d} f(e') - \cfrac{2}{d} f(e_{0})\\
	&=& -\left(\cfrac{2}{d} - \lambda_{1} \right) (f(e) -f(e')) - \cfrac{2}{d}  (f(e) -f(e')) .
	\end{eqnarray*} 
By the symmetry of $e$ and $e'$ and $d(e, e') \geq 1$, we obtain
	\begin{eqnarray*}
	|4/d - \lambda_{1}| |f(e) -f(e')| &\leq& 1 - \kappa\\
	\end{eqnarray*}
Then, the above inequality holds for any edge, which together with (\ref{lip}), implies
	\begin{eqnarray*}
	\lambda_{1} \geq \kappa + \cfrac{4}{d} - 1.
	\end{eqnarray*}
This completes the proof.
	
 \end{proof}

\section{Example}
In this section, we treat some examples of graphs and calculate the eigenvalues of $\mathcal{L'}_{1}$ and the Ricci curvature.

\begin{example}[Complete graph $K_{n}$]
Let $K_{n}$ denote the complete graph with $n$ vertices. While the Ricci curvatures on vertices of $K_{n}$ are equal to $(n-2)/(n-1)$, the Ricci curvature on edges of $K_{n}$ is equal to $1/2$. On the other hand, the first non-zero eigenvalue of $\mathcal{L'}_{1}(K_{n})$ is equal to $n/2(n-2)$.
  \begin{figure}[h]
     \label{Complete} 
     \begin{center} 
     \includegraphics[scale=0.7]{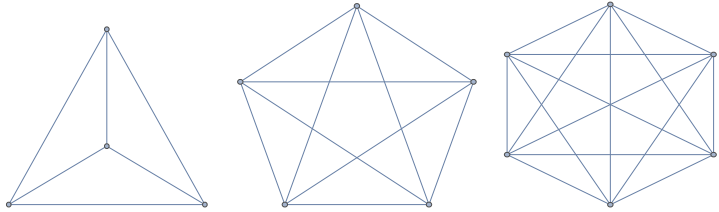}
     \caption{Complete graph}
     \end{center}
    \end{figure}
 \end{example}

\begin{example}[Cycle $C_{n}$]
The Ricci curvatures on vertices and those on edges of the cycle $C_{n}$ with length $n \geq 4$ are all zero. On the other hand, the eigenvalues of $\mathcal{L'}_{1}(C_{4})$ are $0$, $1$ and $2$, and the eigenvalues of $\mathcal{L'}_{1}(C_{5})$ are $0$ and $(5 \pm \sqrt{5})/4$.
\end{example}
  \begin{figure}[h]
     \label{Cycle} 
     \begin{center} 
     \includegraphics[scale=0.6]{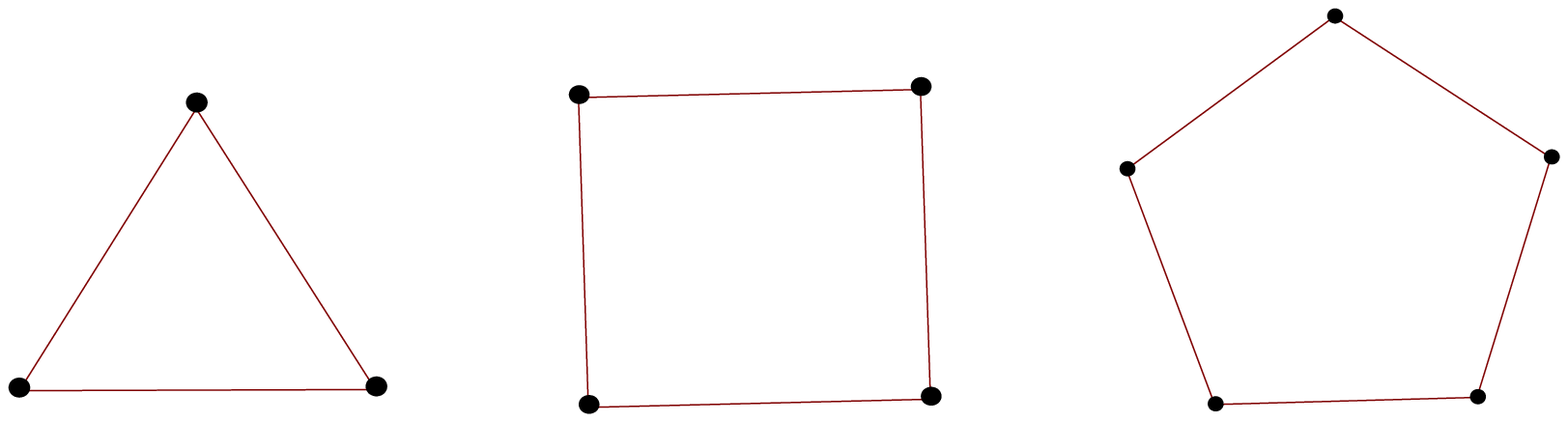}
     \caption{Cycle}
     \end{center}
    \end{figure}

\begin{example}[Complete bipartite graph $K_{n, m}$]
\label{star}
While the Ricci curvatures on vertices of any complete bipartite graph $K_{n, m}$ are all zero, the Ricci curvatures on edges $e=(x, y), e'=(y, z) \in E(K_{n, m})$ is 
\begin{eqnarray*}
\cfrac{d^{V}_{y} - 2}{m + n - 2}.
\end{eqnarray*}
If $n=1$, then $K_{1, m}$ is called the {\em star\ graph}, and the Ricci curvatures on edges of $K_{1, m}$ are $(m-2)/(m -1)$.

\begin{remark}
\label{main0-1}
The first non-zero eigenvalue $\lambda_{1}$ of $\mathcal{L'}_{1}(K_{1, m})$ is equal to $1/(m -1)$. So, if $m \geq 3$, then $\lambda_{1}$ is smaller than the Ricci curvatures. Since the degree of any edge is $m-1$, we have
\begin{eqnarray*}
\kappa + \cfrac{2}{d} - 1 = \cfrac{1}{m-1} = \lambda_{1}.
\end{eqnarray*}
\end{remark}
  \begin{figure}[h]
     \label{Star} 
     \begin{center} 
     \includegraphics[scale=0.8]{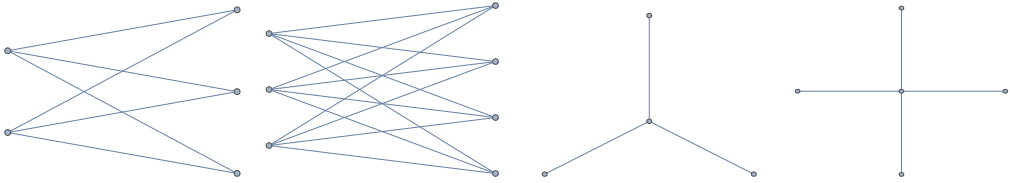}
     \caption{Complete bipartite graph and star graph}
     \end{center}
    \end{figure}
\end{example}
\begin{example}[Tree $T$]
While the Ricci curvature on vertices $x, y \in V(T)$ is $-2(1 - 1/d^{V}_{x} - 1/d^{V}_{y})$, the Ricci curvature on edges $e=(x, y), e'=(y, z) \in E(T)$ is 
\begin{eqnarray*}
\kappa(e, e') = \cfrac{d^{V}_{y}}{\min \left\{d_{e}, d_{e'} \right\}} + \cfrac{2 d^{V}_{y} - 2}{\max \left\{d_{e}, d_{e'} \right\}} - 2.
\end{eqnarray*}
If $T$ is a $d$-regular tree, then the Ricci curvature on edges is $(-d+2)/(2d-2)$.
\end{example}


\section{Relation between a weight function and a weighted graph}
The $i$-Laplacian depends on the weight function $w$. On the other hand, the Ricci curvature depends on the Wasserstein distance. If we figure out a relation between the weight function and the Wasserstein distance, then we obtain an estimate of the non-zero eigenvalues of $i$-Laplacian in more general setting. So, we consider a weighted graph $G$ that has the weight corresponding the weight function. Let $G$ be a weighted undirected graph, and we generalize definitions of the distance, the degree, and a probability measure.
	\begin{definition}
 	For two edges $e_{0}$ and $e_{n}$, let $\left\{e_{i} \right\}_{i=0}^{n}$ be the shortest path connecting $e_{0}$ and $e_{n}$. The {\em distance} between $e_{0}$ and $e_{n}$ is defined by
  		\begin{eqnarray*}
  		d(e_{0}, e_{n}) = \sum_{j=1}^{n} w(v_{j}),
  		\end{eqnarray*}
 	where $v_{j}$ is a vertex such that $e_{j-1}$ and $e_{j}$ are connected by $v_{j}$.
 	\end{definition}
 	\begin{definition}
 	For any edge $e$, the {\em degree} of $e$ is defined by
   		\begin{eqnarray*}
    		d_{e} = \sum_{\bar{e} \in \Gamma(e)} w(\bar{e}).
   		\end{eqnarray*}
 	\end{definition}
 	\begin{definition}
  	For any edge $e$, a probability measure $m_{e}$ is defined by
   		\begin{eqnarray*}
    		m_{e}(\bar{e})=
     			\begin{cases}
      			\cfrac{w(\bar{e})}{d_{e}},& \mathrm{if}\ \bar{e} \in \Gamma(e),\\
      			0,& \mathrm{otherwise}.
     			\end{cases}
   		\end{eqnarray*}
 	\end{definition}
By this definition, the operator $L^{down}$ is expressed as
    	\begin{eqnarray*}
     	(L^{down} f)(e) &=& \sum_{\bar{e} \in \Gamma(e): e \cap \bar{e} = v} \cfrac{w(\bar{e})}{w(v)} \mathrm{sgn}([v], \partial [e]) \mathrm{sgn}([v], \partial [\bar{e}]) f(\bar{e})\\
     	&=& \sum_{\bar{e} \in \Gamma(e): e \cap \bar{e} = v} m_{e}(\bar{e}) \mathrm{sgn}([v], \partial [e]) \mathrm{sgn}([v], \partial [\bar{e}]) \cfrac{d_{e}}{w(v)} f(\bar{e}).
    	\end{eqnarray*}
  We obtain lower and upper bounds of the Ricci curvature on a weighted graph.
 	\begin{theorem}
	\label{weight1}
  	For any edges $e$ and $e'$ with $d(e, e') = 1$, we have
    		\begin{eqnarray*}
     		\kappa(e, e') \geq - 2 \left(1 - \cfrac{w(e')}{d_{e}} - \cfrac{w(e)}{d_{e'}} \right)_{+}
		\end{eqnarray*} 
	\end{theorem} 
  	\begin{theorem}
	\label{weight2}
  	We assume that the weight function on $V$ is a constant, say $w_{0}$. For any two edges $e$ and $e'$ with $d(e, e') = 1$, we have
   		\begin{eqnarray*}
    		\kappa(e, e') \leq \cfrac{w_{\cap}}{\mathrm{max} \left\{d_{e},  d_{e'} \right\}},
   		\end{eqnarray*}
		where  $w_{\cap} := \sum_{\bar{e} \in \Gamma(e) \cap \Gamma(e')} w(\bar{e})$.
  	\end{theorem}
  	\begin{remark}
    Theorems \ref{weight1} and \ref{weight2} are proved in the same way as those of Theorems  \ref{Riccilower} and \ref{Ricciupper}. If $w(v)=1$, $w(e) = 1/d_{e}$, and $|\Gamma(e)|=d_{e}$ for any vertex $v$ and any edge $e$, these results coincide with Theorems \ref{Riccilower} and \ref{Ricciupper}.
  	\end{remark}
  Hereafter, we assume that the weight function on $E$ is a constant, say $w_{1}$. We generalize the estimate of the first non-zero eigenvalue of $\mathcal{L}^{down}_{1}$ as follows.
 	\begin{theorem}
	\label{weight3}
  	Let $\lambda_{1}$ be the first non-zero eigenvalue of $\mathcal{L}^{down}_{1}$. Suppose that $\kappa(e, e') \geq \kappa$ for any $e, e' \in E$ and for a real number $\kappa > 0$. Then, we have
  		\begin{eqnarray*}
   		\lambda_{1} \geq \left\{ d( \kappa - 1 )+ 2 \right\} \cfrac{w_{1}}{w_{0}}.
  		\end{eqnarray*}
 	\end{theorem}
   	\begin{remark}
   	Theorem \ref{weight3} is proved in the same way as that of Theorem \ref{main0}. If $w_{0}=1$, $\Gamma=d$ and $w_{1} = 1/d$, the result coincides Theorem \ref{main0}.
  	\end{remark}

\end{document}